\documentclass{mcom-l}

\usepackage{graphicx} 
\usepackage{extarrows}
\usepackage{latexsym}
\usepackage{amsmath}
\usepackage{amssymb}
\usepackage{amsfonts}
\usepackage{verbatim}
\usepackage{mathrsfs}
\usepackage[modulo]{lineno} 
\usepackage[latin1]{inputenc}
\usepackage{color}
\usepackage{xcolor}
\usepackage[colorlinks,citecolor=blue,urlcolor=blue]{hyperref}
\usepackage{soul}
\usepackage{enumitem}

\newtheorem{theorem}{Theorem}
\newtheorem{proposition}{Proposition}
\newtheorem{corollary}{Corollary}
\newtheorem{lemma}{Lemma}

\theoremstyle{remark}
\newtheorem{remark}{Remark}
\newtheorem{example}{Example}
\newtheorem{assumption}{Assumption}

\numberwithin{equation}{section}

\newcommand{\tr}{\top}

\newcommand{\ee}{\mathbb E}

\newcommand{\pp}{\mathbb P}
\newcommand{\nn}{\mathbb N}
\newcommand{\rr}{\mathbb R}

\newcommand{\AAA}{\mathcal A}
\newcommand{\BB}{\mathcal B}

\newcommand{\LL}{\mathcal L}
\newcommand{\TT}{\mathcal T}

\newcommand{\PP}{\mathcal P}

\newcommand{\OOO}{\mathscr O}
\newcommand{\FFF}{\mathscr F}

\newcommand{\<}{\langle}
\renewcommand{\>}{\rangle}
\allowdisplaybreaks \allowdisplaybreaks[4]

\newcommand{\dd}{\mathrm{d}}

\newcommand{\abs}[1]{\left\lvert #1 \right\rvert}
\newcommand{\norm}[1]{\left\lVert #1 \right\rVert}
\newcommand{\innerproduct}[2]{\left\langle #1,#2 \right\rangle}

\begin{document}

\title[Unique Ergodicity of STM for Monotone SDEs]
{Unique Ergodicity of Stochastic Theta Method for Monotone SDEs driven by Nondegenerate Multiplicative Noise} 
\author{Zhihui LIU}
\address{Department of Mathematics \& National Center for Applied Mathematics Shenzhen (NCAMS), Southern University of Science and Technology, Shenzhen 518055, P.R. China}
\email{liuzh3@sustech.edu.cn}

\author{Zhizhou Liu}
\address{Department of Mathematics, The Hong Kong University of Science and Technology, Hong Kong}
\email{zliugm@connect.ust.hk}
\thanks{The first author is supported by the National Natural Science Foundation of China, No. 12101296, Shenzhen Basic Research Special Project (Natural Science Foundation) Basic Research (General Project), No. JCYJ20220530112814033, and Basic and Applied Basic Research Foundation of Guangdong Province, No. 2024A1515012348.} 




\maketitle

\begin{abstract}
We first establish the unique ergodicity of the stochastic theta method (STM) with $\theta \in [1/2, 1]$ for monotone SODEs, without growth restriction on the coefficients, driven by nondegenerate multiplicative noise.
The main ingredient of the arguments lies in constructing new Lyapunov functions involving the coefficients, the stepsize, and $\theta$ and deriving a minorization condition for the STM.
We then generalize the arguments to the Galerkin-based full discretizations for a class of monotone SPDEs driven by infinite-dimensional nondegenerate multiplicative trace-class noise. 
Applying these results to the stochastic Allen--Cahn equation indicates that its Galerkin-based full discretizations are uniquely ergodic for any interface thickness.
Numerical experiments verify our theoretical results.
\end{abstract}


\section{Introduction}

The long-time behavior of Markov processes generated by stochastic differential equations (SDEs) is a natural and intriguing question and has been investigated in recent decades. 
As a significant long-time behavior, the ergodicity characterizes the case of temporal average coinciding with spatial average, which has a lot of applications in quantum mechanics, fluid dynamics, financial mathematics, and many other fields \cite{DZ96}, \cite{HW19}.
The spatial average, i.e., the mean of a given test function for the invariant measure of the considered Markov process, also known as the ergodic limit, is desirable to compute in practical applications. One has to investigate a stochastic system over long time intervals, which is one of the main difficulties from the computational perspective.  
It is well-known that the explicit expression of the invariant measure for a stochastic nonlinear system is rarely available; exceptional examples are gradient Langevin systems driven by additive noise, e.g., \cite{BV10}, \cite{LV22}.
For this reason, a lot of investigations in the recent decade have been motivated and fascinated by constructing numerical algorithms that can inherit the ergodicity of the original system and approximate the ergodic limit efficiently.

Much progress has been made in the design and analysis of numerical approximations of the desired ergodic limits for SDEs under a strong dissipative condition so that the Markov chains are contractive, see, e.g., \cite{LMW23}, \cite{WL19}, and references therein for numerical ergodicity of backward Euler or Milstein schemes for dissipative SODEs, \cite{LMYY18} for dissipative SODEs with Markovian switching, and \cite{Bre14}, \cite{Bre22}, \cite{BK17}, \cite{BV16}, \cite{CGW20}, \cite{CHS21}, and \cite{Liu24} for approximating the invariant measures via temporal tamed, Galerkin-based linearity-implicit Euler or exponential Euler schemes, and high order integrator for parabolic SPDEs driven by additive noise.
See also \cite{MSH02} and \cite{GM23}, \cite{HM06} for numerical ergodicity of backward Euler scheme and its versions for monotone SODEs and spectral Galerkin approximation for 2-D stochastic Navier--Stokes equations, respectively, both driven by additive degenerate noise.

We note that most of the above works of literature focus on the numerical ergodicity of strong dissipative SDEs driven by additive noise; the numerical ergodicity for weak dissipative SDEs in the multiplicative noise case is more subtle.   
This question on the unique ergodicity of numerical approximations for monotone SDEs in finite and infinite-dimensional settings motivates the present study. 
Our main aim is to establish the unique ergodicity of the STM scheme \eqref{stm} with $\theta \in [1/2, 1]$, including the numerical schemes studied in \cite{HMS02}, \cite{MSH02}, and \cite{QLH14} for monotone SODEs, without growth restriction on the coefficients, driven by nondegenerate multiplicative noise. 

It is not difficult to show that $V(\cdot) = \abs{\cdot}^2+1$ is a natural Lyapunov function of the considered monotone \eqref{sde}.
However, it was shown in \cite{HJK11} that the Euler--Maruyama scheme (i.e., \eqref{stm} with $\theta=0$) applied to Eq.~\eqref{sde} with superlinear growth coefficients would blow up in $p$-th moment for all $p \ge 2$.
Therefore, $\abs{\cdot}^2+1$ is not an appropriate Lyapunov function of this scheme in the setting of the present study, and we mainly focus on the case $\theta \in [1/2, 1]$.
By exploring the monotone structure of the coefficients and martingale property of the driven Wiener process, we construct several Lyapunov functions (see, e.g., \eqref{lya} and \eqref{lya-mid} in Theorem \ref{tm-lya} and Corollary \ref{cor-lya-mid}, respectively) for the STM, which involve both the coefficients, the stepsize, and $\theta$.
In combination with a minorization condition followed by deriving the irreducibility and the existence of a jointly continuous density for the transition kernel (see Proposition \ref{prop-minor}), we conclude the geometric ergodicity of the proposed STM scheme \eqref{stm}.
  
Then we generalize our methodology to Galerkin-based fully discrete schemes (see \eqref{die-g}), which have been studied in \cite{CHS21}, \cite{GM09}, \cite{JR15}, \cite{Liu22}, and \cite{LQ21} for monotone SPDEs with polynomial growth coefficients driven by infinite-dimensional nondegenerate multiplicative trace-class noise (see Theorem \ref{tm-dieg}).  
Applying these results to the stochastic Allen--Cahn equation (SACE) driven by nondegenerate noise indicates that its Galerkin-based full discretizations are uniquely ergodic, respectively, for any interface thickness  (see Theorem \ref{tm-ac}).

The paper is organized as follows.
In Section \ref{sec2}, we give the principal assumptions on monotone SODEs and recall the ergodic theory of Markov chains that will be used throughout.
The STM's Lyapunov structure and minorization condition are explored in Section \ref{sec3}.
In Section \ref{sec4}, we generalize the arguments in Section \ref{sec3} to monotone SPDEs, including the SACE.
The theoretical results are validated by numerical experiments in Section \ref{sec5}.
We include some discussions in the last section.

\section{Assumptions on SODEs and Useful Facts from Markov Chain Theory}
\label{sec2}

In this section, we present the required main assumptions and recall the general ergodic theory of Markov chains that will be used throughout the paper. We denote the set $\{0,1,\dots\}$ by $\nn$ and the set $\{1,2,\dots\}$ by $\nn_+$.

\subsection{Main Assumptions on SODEs}
\label{sec2.1}

Let us first consider the $d$-dimensional SODE
\begin{equation}\label{sde}
    \dd X(t) = b(X(t))\dd t + \sigma(X(t)) \dd W(t), ~ t \ge 0, \tag{SODE}
\end{equation}
driven by an $\rr^m$-valued Wiener process $W$ on a complete filtered probability space $(\Omega, \FFF, \mathbb{F}:=(\FFF(t))_{t\geq 0}, \pp)$, where $b: \rr^d \rightarrow \rr^d$ and $\sigma:\rr^d \rightarrow \rr^{d \times m}$ are continuous functions.

Our main focus is on the invariant measure of the Markov chain generated by the following stochastic theta method (STM) with $\theta \in [0, 1]$:
\begin{equation}\label{stm}
    X_{j+1} = X_j  + (1-\theta) b(X_j ) \tau 
    + \theta b(X_{j+1}) \tau + 
    \sigma(X_j ) \delta_j W,   \tag{STM}
\end{equation}
where $\tau \in (0, 1)$ is a fixed step-size, $\delta_j W:=W(t_{j+1})-W(t_j)$ with $t_j=j \tau$, $j\in \nn$. 
When $\theta=0,1/2,1$, it is called the Euler--Maruyama scheme, the trapezoid scheme, and the backward Euler method, respectively. We shall write $\FFF_j:= \FFF(t_j)$ for notational convenience.

The following coupled monotone condition is our first primary condition on the drift $b$ and diffusion $\sigma$ in Eq.~\eqref{sde}.
Throughout, we denote by $\<\cdot,\cdot\>$ the standard inner product in $\rr^d$ or $\rr^m$, by $|\cdot|$ its induced norm, and by $\norm{\cdot}$ the Hilbert--Schmidt norm in $\rr^{d \times m}$. 
 
\begin{assumption} \label{A1}
  There exists a constant $L_1\in \rr$ such that
  \begin{equation} \label{mon}
      2\innerproduct{b(x)-b(y)}{x-y} + \norm{\sigma(x)-\sigma(y)}^2 
      \leq L_1 \abs{x-y}^2, \quad \forall~ x,y\in \rr^d.
  \end{equation} 
\end{assumption}

Under Assumption \ref{A1} and certain integrability condition on $b$ and $\sigma$, one can show the existence and uniqueness of the $(\FFF_t)_{t \ge 0}$-adapted solution to Eq. \eqref{sde}, see, e.g., \cite[Theorem 3.1.1]{LR15}.
The following lemma gives the analogous results for our numerical scheme \eqref{stm}. We shall always let this assumption hold to ensure the scheme we discuss is a well-defined Markov chain.

\begin{lemma}\label{lem:well-define}
  Let Assumption \ref{A1} hold. 
  Then STM scheme \eqref{stm} applied to Eq.~\eqref{sde} can be uniquely solved when $L_1 \theta \tau<2$. Moreover, $(X_n)_{n\in \nn}$ is adapted to $(\FFF_n)_{n\in \nn}$ and enjoys homogenous Markov property.
\end{lemma}

\begin{proof}
  Define an auxiliary function $\hat b : \rr^d \to \rr^d$ by
    \begin{align} \label{hatb}
      \hat b(x) := x- \theta b(x) \tau, \quad x \in \rr^d.
    \end{align}
    Our numerical scheme \eqref{stm} can be rewritten as 
  \begin{equation*}
    \hat b(X_{j+1}) 
    = X_j + (1-\theta)b(X_j^h) \tau +\sigma(X_j )\delta_j W, \quad j \in \nn.
  \end{equation*}
  We want to show the invertibility of $\hat b$ through uniform monotonicity. Note that 
  \begin{align*}
      \<x-y, \hat b(x)-\hat b(y)\> 
      &= \innerproduct{x-y}{x-y-(b(x)-b(y))\theta \tau} \\
      &= |x-y|^2 - \innerproduct{x-y}{b(x)-b(y)} \theta \tau \\
      & \geq (1-L_1 \theta \tau/2 )|x-y|^2.
  \end{align*}
  When $L_1 \theta \tau<2$, we see that $\hat b$ defined in \eqref{hatb} is uniformly monotone and thus invertible (see, e.g., \cite[Theorem C.2]{SH96}), from which we can conclude the unique solvability of \eqref{stm}.

  As $X_{j+1}$ can be written as a function of $X_j$ and $\delta_j W= W(t_{j+1})-W(t_j)$ (which is $\FFF_{j+1}$-measurable), one can easily show by induction that $(X_n)_{n\in \nn}$ is adapted to $(\FFF_n)_{n\in \nn}$. 
  Finally, the homogenous Markov property is evident as $\delta_j W$ is stationary distributed and independent of $X_j$.
\end{proof}

Our second assumption is the following coupled coercive condition.

\begin{assumption} \label{A2}
  There exist two positive constants $L_2$ and $L_3$ such that
  \begin{equation} \label{coe}
      2\innerproduct{b(x)}{x} + \norm{\sigma(x)}^2 \leq  L_2 -
       L_3 \abs{x}^2, \quad \forall~ x\in \rr^d.
  \end{equation} 
\end{assumption}

We also need the following conditions for nondegeneracy and regularity. Here and after, $^\tr$ denotes the corresponding transpose of a matrix.

\begin{assumption} \label{A3}
\begin{enumerate}
\item
For each $x \in \rr^d$, the matrix $\sigma(x) \sigma(x)^\tr \in \rr^{d \times d}$ is positive definite. 

\item
$b$ and the inverse $[\sigma \sigma^\tr]^{-1}$ of $\sigma \sigma^\tr$ are continuously differentiable.
\end{enumerate}
\end{assumption}
 
Our assumptions are motivated by the below example.

\begin{example}\label{ex1}
Consider $d=m=1$ and 
\begin{align*}
b(x)=x-x^3, \quad \sigma(x)=\sqrt{x^2+1}, \quad x \in \rr. 
\end{align*}
Note that our choice of $\sigma$ is Lipschitz continuous with Lipschitz constant 1, so  
\begin{align*}
  &2 \innerproduct{b(x)-b(y)}{x-y}+\|\sigma(x)-\sigma(y)\|^2 \\
  =& 2(x-y)^2 (1-x^2-xy-y^2)+|\sqrt{x^2+1}-\sqrt{y^2+1}|^2 \\
  \leq &  \abs{x-y}^2 (3-2x^2-2xy-2y^2) \le 3\abs{x-y}^2,
\end{align*}
as $x^2+xy+y^2 \geq 0$, for all $x, y \in \rr$. This shows Assumption \ref{A1} holds with $L_1=3$. We now verify the second assumption:
\begin{equation*}
  2\innerproduct{b(x)}{x} + \|\sigma(x)\|^2= 2x^2-2x^4+x^2+1 \leq 3-x^2.
\end{equation*}
This shows Assumption \ref{A2} with $L_2=3$, $L_3=1$. Finally, it is clear that Assumption \ref{A3} holds. 
\end{example} 

\begin{remark}
  The above example also shows that even if we have a negative coefficient $-L_3$ in the growth condition \eqref{coe}, we can still consider \eqref{mon} with a positive coefficient $L_1$. These different dissipativity strengths would make the numerical ergodicity analysis very different; see the recent work \cite{LMW23} for the strongly dissipative case.
\end{remark}

Under assumptions very similar to us (Assumption \ref{A1}, \ref{A2} and \ref{A3}), Eq.~\eqref{sde} is ergodic, see, e.g., \cite{RWZ10} and \cite{Zha09} for the use of coupling methods based on a changing measure technique in combination of Girsanov theorem. 
In this paper, we focus on the numerical scheme, and such a Girsanov theorem does not exist in the discrete-time model.
Instead, we will utilize the general ergodic theory of Markov chains.

\subsection{Preliminaries on Ergodicity of Markov Chains}
\label{sec2.2}

Let $(H, \|\cdot\|)$ be a Hilbert space equipped with Borel $\sigma$-algebra $\BB(H)$, which will be chosen to be either $\rr^d$ in Section \ref{sec3}, or the finite-dimensional approximate spaces $V_h$ and $V_N$ in Section \ref{sec4}. 
Denote by $\mathcal N(a, b)$ and $\mu_{a, b}$ the normal distribution and Gaussian measure in $H$, respectively, with mean $a \in H$ and variance operator $b \in \LL(H)$ (the space of positive symmetric trace class operator on $H$). 

Let $(X^{x_0}_n)_{n \in \nn}$ (we will omit to write the dependence on $x_0$ for brevity) be an $H$-valued homogenous Markov chain on $(\Omega, \FFF, \pp)$ with initial state $x_0 \in H$ and transition kernel $P(x_0, B):=\pp(X_{n+1}\in B \mid X_n=x_0)$, $B \in \BB(H)$, $n\in \nn$. 
For $x \in H$, $B \in \BB(H)$, the $n$-step transition kernel is denoted by $P^n(x, B)$ (in particular, $P^1(x, B)=P(x, B)$) and defined inductively as
\[
P^n (x,B) := \int_H  P(y, B) P^{n-1}(x, \dd y), \quad n \in \nn_+.
\] 
We also use the same $P$ to denote the bounded linear operator in the set $\BB_b(H)$ of all bounded, measurable functions on $H$ if there is no confusion:
\[
P\phi(x) := \int_H \phi(y) P(x, \dd y), \quad x \in H,
\]
and $P_n \phi(x)= \int_H \phi(y) P^n(x, \dd y)$ for $n \ge 1$ and $x \in H$; in particular, $P_1=P$. 

To introduce the ergodic theory of Markov chains, we say that a probability measure $\mu$ on $\BB(H)$ is called \emph{invariant} for the Markov chain $(X_n)_{n \in \nn}$ or its transition kernel $P$, if 
\begin{align*}
\int_H P \phi(x) \mu({\rm d}x)=\mu(\phi):=\int_H \phi(x) \mu({\rm d}x),
\quad \forall~ \phi \in \BB_b(H).
\end{align*} 
This is equivalent to $\int_H P(x, A) \mu({\rm d}x)=\mu(A)$ for all $A \in \BB(H).$
An invariant (probability) measure $\mu$ is called \emph{ergodic} for $(X_n)_{n \in \nn}$ or $P$, if 
\begin{align} \label{df-erg}
\lim_{m \to \infty} \frac1m \sum_{k=0}^m P_k \phi(x)=\mu(\phi) 
\quad \text{in}~ L^2(H; \mu),\quad \forall~ \phi \in L^2(H; \mu).
\end{align}
It is well-known that if $(X_n)_{n \in \nn}$ admits a unique invariant measure, then it is ergodic; in this case, we call it uniquely ergodic.

We introduce three Lyapunov-type sufficient conditions (ergodic theorems) for unique ergodicity, which we will use in the following sections. We start by recalling the concept of \emph{$m$-small sets}. A set $C\in \BB(H)$ is called an \emph{$m$-small set} if there exists an integer $m>0$ and a non-trivial measure (i.e., $\nu(H)>0$) $\nu$ on $\BB(H)$ such that  
\begin{align*}
    P^m(x,B)\geq \nu (B), \quad \forall~x\in C,~ B\in \BB(H).
\end{align*} 
If the $m$-small set exists, we often say that the \emph{minorization condition} is satisfied for $P^m$. We say a Markov chain $(X_n)_{n \in \nn}$ is \emph{(open set) irreducible} if $P^n(x, A)>0$ for any $x \in H$, any non-empty open set $A$ in $H$ and some $n\in \nn_+$; it is called \emph{strong Feller} if $P(\cdot, A)$ is lower continuous for any Borel set $A \in \BB(H)$. We say the transition kernel $P$ of the Markov chain $(X_n)_{n\in \nn}$ is regular if the family of probability measures $\{P(x, \cdot): x\in H\}$ are mutually equivalent.

The \emph{first} classic ergodic theorem in Markov chain theory we needed is that if a Markov chain $(X_n)_{n\in \nn}$ is irreducible and strong Feller satisfying minorization condition for $P^m$, and there exists a positive function $V: H \to [1,\infty)$ and a constant $c$ such that
\begin{equation}\label{lya1}
  PV(x)-V(x) \leq -1 + c \chi_C(x), \quad \forall~ x \in H,
\end{equation} 
for some $m$-small set $C$, then the chain is uniquely ergodic. The \emph{second} ergodic theorem yields that with the same conditions as the first but \eqref{lya1} replaced by a stronger Lyapunov condition: there exists $\rho\in (0,1)$, $\kappa\geq 0$ and a positive function $V: H\to[1,\infty)$ with $\lim_{\|x\|\to\infty} V(x)=\infty$ such that
\begin{equation}\label{lya2}
    PV \leq \rho V +\kappa,
\end{equation}
the chain is geometrically ergodic, i.e., there exists $r\in (0,1)$ and $K \in (0,\infty)$ such that for all measurable $\phi$ with $|\phi|\leq V$, and the unique invariant measure $\mu$ of the chain, 
\[
\left| \ee \phi(X_n^{x_0})- \int_H \phi \dd \mu \right| \leq K r^n  V(x_0),
\quad x_0 \in \rr^d.
\]
In addition, \eqref{lya2} ensures the existence of an invariant measure. 
The \emph{third} ergodic theorem yields that given the existence of invariant measure, if the transition kernel $P$ of the chain is regular, then the chain is uniquely ergodic. 

We close this subsection by remarking that for the irreducible strong Feller chain in a finite-dimensional Hilbert space we considered, the condition $\lim_{\|x\|\to\infty}V(x)=\infty$ ensures that any sublevel set $\{V\leq C\}$ of $V$ is a compact $m$-small set (for some $m \in \nn_+$). We refer to \cite{MT09} for more details on Markov chain theory.

\section{Geometric Ergodicity of STM for Monotone SODEs}
\label{sec3}

Denote by $\AAA$ the infinitesimal generator of \eqref{sde}, i.e.,
\begin{align*}
\AAA :=\sum_{i=1}^d b^i \frac{\partial}{\partial x_i}
+\frac12 \sum_{i,j=1}^d (\sigma \sigma^\tr)_{ij} \frac{\partial^2}{\partial x_i\partial x_j}.
\end{align*}
Under Assumption \ref{A2}, it is not difficult to show that 
\begin{align*}
\AAA \abs{X(t)}^2=2\innerproduct{b(X(t))}{X(t)} + 
\norm{\sigma(X(t))}^2 \leq  L_2 - L_3 \abs{X(t)}^2, \quad t \ge 0.
\end{align*}
Hence by utilizing Ito's formula $V(\cdot) = \abs{\cdot}^2+1$ is a natural Lyapunov function of \eqref{sde}; see e.g. \cite[(2.2)]{MSH02} for more details.
However, it was shown in \cite{HJK11} that for \eqref{stm} with $\theta=0$ applied to \eqref{sde} with superlinear growth coefficients, $\lim_{n \to \infty} \ee |X_n|^p=\infty$ for all $p \ge 2$ so that $V(x) = \abs{x}^p+1$ is not an appropriate Lyapunov function of this scheme as the Lyapunov condition implies uniform moment stability (see Remark \ref{rk:3.1}(i)).

In this section, we first construct a Lyapunov function and then derive the irreducibility and existence of continuous density for \eqref{stm} with $\theta \in [1/2, 1]$. The case $\theta\in [0,1/2)$ is not covered.

\subsection{Lyapunov structure}
\label{sec3.1}

We begin with the following technical inequality, which will be helpful to explore the Lyapunov structure of \eqref{stm}.
It is a straightforward generalization of \cite[Lemma 3.3]{QLH14}. Intuitively, the inequality says that we can throw away the term involving $\|\sigma(x)\|$ in the left-hand side by sacrificing a bit on the coefficients.

\begin{lemma}\label{lm-in-stm}
  Let Assumption \ref{A2} hold.
  Then for any $\rho,\beta\in \rr$ with $\rho\geq \beta\geq 0$,
  it holds with $C = \frac{1+\beta L_3}{1+\rho L_3}$ that for all $x \in \rr^d$   
  \begin{equation}\label{in-stm}
          \abs{x-\beta b(x)}^2 + \beta  L_2 
          +(C \rho-\beta) \|\sigma(x)\|^2
          \le C (\abs{x-\rho b(x)}^2 + \rho L_2)
      \end{equation}   
\end{lemma}

\begin{proof}
Let $x \in \rr^d$. By separating $x$ and $b(x)$, we have
    \begin{align*}
        |x-&\beta b(x)|^2 + \beta L_2 +(C \rho-\beta) \|\sigma(x)\|^2 -  C (\abs{x-\rho b(x)}^2 + \rho L_2)\\
        = & (1-C)\abs{x}^2 + ( C\rho-\beta) [2\innerproduct{x}{b(x)}- L_2+ \|\sigma(x)\|^2] 
        + (\beta^2- C\rho^2)\abs{b(x)}^2.
    \end{align*}
    Then it follows from \eqref{coe} that the above equality can be bounded by 
    \begin{align*}
        [(1-C)-(C\rho-\beta) L_3]\abs{x}^2 + 
        (\beta^2-C\rho^2)\abs{b(x)}^2,
    \end{align*}
    from which we conclude \eqref{in-stm} by noting that $1-C=( C\rho-\beta) L_3$ and $\beta^2-C\rho^2 \le 0$ with $C = \frac{1+\beta L_3}{1+\rho L_3}$.
\end{proof}

Now, we state and prove the following strong Lyapunov structure for \eqref{stm} with $\theta\in (1/2,1]$, which guarantees the existence of the invariant measure.

\begin{theorem} \label{tm-lya}
  Let Assumption \ref{A1} and \ref{A2} hold and $\theta\in (1/2,1]$. 
  For any $\tau \in (0, 1)$ with $L_1 \theta \tau<2$, and any constant $\lambda \in (0,  2\theta-1]$, the function $V_\theta:\rr^d \to [1,\infty)$ defined by 
\begin{equation}\label{lya}
      V_\theta(x) := \abs{x-(1-\theta+\lambda)b(x) \tau}^2 +  
      (2\theta-1-\lambda) \norm{\sigma(x)}^2 \tau +1 \quad x \in \rr^d, 
  \end{equation}
  satisfies $\lim_{x\to\infty}V_\theta(x)=\infty$ and \eqref{lya2} with $\rho := \frac{1+(1-\theta)  L_3 \tau}{1+(1-\theta+\lambda)  L_3 \tau} \in (0, 1)$ and $\kappa := L_2 \tau +(1-\rho)>0$, i.e., for $n \in \nn$,
  \begin{equation}\label{lya+}
    \begin{split}
      \ee [V_\theta(X_{n+1})\mid \FFF_{n}] \leq & 
      \frac{1+(1-\theta)  L_3 \tau}{1+(1-\theta+\lambda)  L_3 \tau} V_\theta (X_n)
      +  L_2 \tau +  \frac{\lambda L_3 \tau}{1+(1-\theta+\lambda)  L_3 \tau}.
    \end{split}
  \end{equation}  
\end{theorem}

\begin{proof}
To simplify the notation in the following proof, we write as $b_n:=b(X_n)$ and $\sigma_n:=\sigma(X_n)$ and define 
\begin{align*}
F_n:=\abs{X_{n}-(1-\theta+\lambda)b_n \tau}^2 + 
L_2 (1-\theta+\lambda) \tau.
\end{align*}

Remember that we aim to obtain a recursive relation for $(V_\theta(X_n))_{n\in \nn}$. Firstly, we rearrange the terms to prepare for the substitution of $X_{n+1}$ by $X_n$ via their recursive relation: 
\begin{align*}
  F_{n+1}
  = & \abs{X_{n+1}-\theta b_{n+1} \tau +  (2\theta-1-\lambda)b_{n+1} \tau}^2 +    L_2 (1-\theta+\lambda) \tau \\
  = & \abs{X_{n+1}-\theta b_{n+1} \tau}^2 + 
  2(2\theta-1-\lambda) [\innerproduct{X_{n+1}}{b_{n+1}} \tau - 
  \theta \abs{b_{n+1}}^2 \tau^2] \\
    & \qquad  + (2\theta-1-\lambda)^2 \abs{b_{n+1}}^2 \tau^2 + 
      L_2 (1-\theta+\lambda) \tau \\
  = & \abs{X_{n+1}-\theta b_{n+1} \tau}^2 + 
  2(2\theta-1-\lambda)\innerproduct{X_{n+1}}{b_{n+1}} \tau  \\
  & \qquad + (\lambda+1-2\theta)(\lambda+1)\abs{b_{n+1}}^2 \tau^2 
  +  L_2 (1-\theta+\lambda) \tau.
\end{align*} 
We can substitute \eqref{stm} in the first term on the right-hand side of the above equality to obtain
\begin{align*}
  \abs{X_{n+1}-\theta b_{n+1} \tau}^2 
  = & \abs{X_n + (1-\theta) b_n \tau  + \sigma _n \delta_n W}^2 \\
  = & \abs{X_n + (1-\theta)b_n \tau}^2 + 2 \innerproduct{X_n + (1-\theta)b_n \tau}{\sigma_n
  \delta_n W} + \abs{\sigma_n\delta_n W}^2 \\
  = & \abs{X_n - (1-\theta)b_n \tau}^2 + 
  4(1-\theta) \innerproduct{X_n}{b_n} \tau  \\
  & \qquad + 2 \innerproduct{X_n + (1-\theta)b_n \tau}{\sigma_n
  \delta_n W} + \abs{\sigma_n\delta_n W}^2.
\end{align*}
It follows that
\begin{equation}\label{eq-stm}
  \begin{split}
    F_{n+1}= & \abs{X_n - (1-\theta)b_n \tau}^2 + 4(1-\theta) 
  \innerproduct{X_n}{b_n} \tau \\
  & \qquad + 2(2\theta-1-\lambda)\innerproduct{X_{n+1}}{b_{n+1}} \tau \\
  & \qquad + (\lambda+1-2\theta) (\lambda+1) \abs{b_{n+1}}^2 \tau^2 
  + \norm{\sigma_n}^2 \tau \\
  & \qquad +  L_2 (1-\theta +\lambda) \tau + O_n,
  \end{split}
\end{equation}
where 
\begin{align*}
O_n = 2\innerproduct{X_n + (1-\theta)b_n \tau}{\sigma_n
\delta_n W} +\abs{\sigma_n\delta_n W}^2 
- \norm{\sigma_n}^2 \tau.
\end{align*}
Note that $\delta_n W$ is independent of $\FFF_n$ so that by the normal distributed natural of $\delta_n W$, we see that $\ee[O_n \mid \FFF_n]=0$.

Following \eqref{eq-stm} and Lemma \ref{lm-in-stm} with $\rho=(1-\theta+\lambda) \tau$, $\beta=(1-\theta) \tau$, $C=N_\tau:=\frac{1+(1-\theta)  L_3 \tau}{1+(1-\theta+\lambda)  L_3 \tau} \in (0, 1)$, and $C \rho-\beta=B_\tau:=\frac{\lambda \tau}{1+(1-\theta+\lambda)  L_3 \tau}$, and discarding the term $(\lambda+1-2\theta) (\lambda+1) \abs{b_{n+1}}^2 \tau^2$ under the assumption $0<\lambda \le 2\theta-1$, we have
      \begin{align*}
          F_{n+1}  
          \leq & \abs{X_n-(1-\theta)b_n \tau}^2 + 
          (1-\theta) L_2 \tau +  
          4(1-\theta)\innerproduct{X_n}{b_n} \tau + 
          \norm{\sigma_n}^2 \tau \\
              & \qquad + 2(2\theta-1-\lambda)\innerproduct{X_{n+1}}
              {b_{n+1}} \tau  +  L_2 \lambda \tau + O_n \\
          \leq & N_\tau F_n -B_\tau \|\sigma_n\|^2 \tau+ 4(1-\theta)
          \innerproduct{X_n}{b_n} \tau + \norm{\sigma_n}^2 \tau \\
              & \qquad + 2(2\theta-1-\lambda)
              \innerproduct{X_{n+1}}{b_{n+1}} \tau  +  L_2 \lambda \tau + O_n \\
          \leq & N_\tau F_n + 2(1-\theta) ( L_2- L_3 |X_n|^2- \norm{\sigma_n}^2) \tau + (1-B_\tau)\norm{\sigma_n}^2 \tau \\
          &  \qquad +   (2\theta-1-\lambda)[ L_2- L_3 |X_{n+1}|^2-\norm{\sigma_{n+1}}^2] \tau  + \lambda L_2 \tau + O_n.
      \end{align*} 
Taking the conditional expectation $\ee [\cdot\mid \FFF_n]$ on both sides of the above inequality, using the fact that $\ee [O_n \mid \FFF_n]=0$, and discarding the terms $-2(1-\theta)  L_3 |X_n|^2$ and $-(2\theta-1-\lambda) L_3 |X_{n+1}|^2$ under the assumption $\lambda+1-2\theta \leq 0$, $n \in \nn$, we have  
      \begin{equation}\label{eq:Lya-b}
        \begin{aligned}
          \ee [F_{n+1} + &(2\theta-1-\lambda) \norm{\sigma_{n+1}}^2 \tau
            \mid \FFF_n]\\
            \leq & N_\tau  F_n +  [1-B_\tau-2(1-\theta)]\norm{\sigma_n}^2 \tau +
             L_2 \tau \\
          \leq & N_\tau [F_n + (2\theta-1-\lambda)  
          \norm{\sigma_n}^2 \tau] +  L_2 \tau,
        \end{aligned}
      \end{equation}
      where in the last inequality, we have used 
      \begin{align*}
      1-B_\tau-2(1-\theta) < N_\tau (2\theta-1-\lambda), 
      \end{align*}
      which is ensured by the estimate
      \begin{align*} 
      \frac{1+(1+\lambda)  L_3 \tau}{1+(1-\theta+\lambda)  L_3 \tau}>1.
      \end{align*} 
Therefore, the relation \eqref{lya+} follows from \eqref{eq:Lya-b} by subtracting the same constant $L_2 (1-\theta+\lambda) \tau$ from both sides.

Finally, we point out that the function $V_\theta: \rr^d \to [0,\infty)$ defined by \eqref{lya} satisfies $\lim_{x\to\infty}V_\theta(x)=\infty$ so that it is a Lyapunov function.
      Indeed, 
      \begin{align} \label{pf-lya}
          V_\theta(x) \geq & \abs{x}^2 -2(1-\theta+\lambda)
          \innerproduct{x}{b(x)} \tau + 
          (2\theta-1-\lambda) \norm{\sigma(x)}^2 \tau   \nonumber \\
          \geq & \abs{x}^2 -(1-\theta+\lambda)
          [ L_2 - L_3\abs{x}^2 - \norm{\sigma(x)}^2] \tau  + (2\theta-1-\lambda) \norm{\sigma(x)}^2 \tau \nonumber  \\
          = & [1+(1-\theta+\lambda) L_3 \tau]\abs{x}^2 + 
          \theta \norm{\sigma(x)}^2 \tau-L_2(1-\theta+\lambda) \tau \nonumber  \\
          \ge & [1+(1-\theta+\lambda) L_3 \tau]\abs{x}^2-L_2(1-\theta+\lambda) \tau,  
      \end{align}
      which tends to $\infty$ as $x \to \infty$.
  \end{proof}

\begin{remark}\label{rk:3.1}
\begin{enumerate}[label=(\roman*)]
\item
It follows from the estimation \eqref{lya+} that 
\begin{align*} 
\ee V_\theta(X_n) \leq \rho^n \ee V_\theta(X_0)+ \frac{L_2 \tau}{1-\rho}
\leq \ee V_\theta(X_0)+ \frac{L_2 \tau}{1-\rho}, \quad n \in \nn,
\end{align*} 
with $\rho=\frac{1+(1-\theta)  L_3 \tau}{1+(1-\theta+\lambda)  L_3 \tau}$, 
which shows the uniform moment stability of $(V_\theta(X_n))_{n \ge 1}$.
          
\item
For the backward Euler scheme, a special case of \eqref{stm} with $\theta=1$, one can take the Lyapunov function (with $\lambda=1$) as  
\begin{align*}
  V_1(x)=\abs{x- b(x) \tau}^2+1, \quad x \in \rr^d, 
\end{align*}  
which coincides with the choice in \cite[Section 8.2]{MSH02} for dissipative problems.
\end{enumerate}
\end{remark}

For $\theta=1/2$, the trapezoid scheme, we have the following weak Lyapunov structure, which is sufficient, in combination with the irreducibility and strong Feller property in the next part, to derive the unique ergodicity.

\begin{corollary} \label{cor-lya-mid}
Let Assumption \ref{A1} and \ref{A2} hold.
For $\tau \in (0, 1)$, the function $V_{1/2}:\rr^d \to [1,\infty)$ defined by  
\begin{align} \label{lya-mid}
    V_{1/2}(x)=|x- \frac12 b(x) \tau|^2 +1, \quad x \in \rr^d, 
\end{align} 
satisfies 
\begin{equation}\label{lya-mid+}
\ee [V_{1/2}(X_{n+1})\mid \FFF_{n}] \leq V_{1/2}(X_n) +  L_2 \tau- L_3 |X_n|^2\tau, \quad n \in \nn.
\end{equation} 

Consequently, \eqref{lya1} holds with $c:=L_2 \tau+1$ and $C:=\{x \in \rr^d: ~ |x|^2 \le \frac{L_2 \tau+1}{L_3 \tau}\}$ which is a compact $m$-small set.
\end{corollary}
  
\begin{proof}
We inherit the notations from the proof of Theorem \ref{tm-lya} with $\theta=1/2$ and $\lambda=0$. Then, following the arguments there, we can define
\begin{align*}
F_n=|X_{n}-\frac12 b_n \tau|^2 + \frac12 L_2 \tau, 
\end{align*}
and have
\begin{align*} 
  F_{n+1}= F_n + 2 \innerproduct{X_n}{b_n} \tau
  + \norm{\sigma_n}^2 \tau + O_n,
\end{align*}
where $O_n = 2\<X_n + \frac12 b_n \tau, \sigma_n
\delta_n W\> +\abs{\sigma_n\delta_n W}^2 
- \norm{\sigma_n}^2 \tau$ satisfies $\ee [O_n \mid \FFF_n]=0$ for $n \in \nn$.

By \eqref{coe}, we have  
\begin{align*}
    F_{n+1}  
    \leq & F_n + L_2 \tau - L_3 |X_n|^2 \tau + O_n.
\end{align*} 
Taking the conditional expectation $\ee [\cdot\mid \FFF_n]$ on both sides of the above inequality and using $\ee [O_n \mid \FFF_n]=0$, we have  
\begin{align*}
    \ee [F_{n+1} \mid \FFF_n] \leq  F_n +  L_2 \tau - L_3 |X_n|^2,
\end{align*} 
and thus shows \eqref{lya-mid+} with $V_{1/2}$ defined in \eqref{lya-mid}. 

Finally, the inequality \eqref{lya1} with $c=L_2 \tau+1$ and $C:=\{x \in \rr^d: ~ |x|^2 \le \frac{L_2 \tau+1}{L_3\tau}\}$ follows from \eqref{lya-mid+} by direct calculations. We know that $C$ is a compact $m$-small set from the remark at the end of Section \ref{sec2.2}. 
\end{proof}

  \subsection{Minorization condition and unique ergodicity of STM}
\label{sec3.2}

To derive the required minorization condition for \eqref{stm}, in this part, we denote by $P(x,A)$, $x \in \rr^d$, $A \in \BB(\rr^d)$, the transition kernel of the Markov chain $(X_n)_{n \in \nn}$ generated by \eqref{stm}.
Recall that for the homogenous Markov chain $(X_n)_{n \in \nn}$, we have for any $n \in \nn$, 
\begin{align*}
P(x,A)=\pp(X_{n+1}\in A\mid X_n=x), \quad x \in \rr^d, ~ A \in \BB(\rr^d).
\end{align*}

Let $x \in \rr^d$, $A \in \BB(\rr^d)$ be a non-empty open set, and set $\widetilde{b}(x):=x+(1-\theta)b(x) \tau$. 
Then, the above Markov property yields that
\begin{align} \label{pxa}
    P(x, A) 
=\pp(\widetilde{b}(x) + \sigma(x) \delta_n W \in \hat{b}(A)) 
= \mu_{\widetilde{b}(x), \sigma(x)\sigma(x)^\tr \tau}(\hat{b}(A)),
\end{align} 
as $\widetilde{b}(x) +\sigma(x) \delta_n W \sim \mathcal N(\widetilde{b}(x), \sigma(x)\sigma(x)^\tr \tau)$, 
where $\hat b$ is defined in \eqref{hatb}. Due to the nondegeneracy of $\sigma$ in Assumption \ref{A3}(i), the Gaussian measure $\mu_{\widetilde{b}(x), \sigma(x)\sigma(x)^\tr \tau}$ is nondegenerate.
As we all know, in the finite-dimensional case, all non-degenerate Gaussian measures $\{\mu_{\widetilde{b}(x), \sigma(x)\sigma(x)^\tr \tau}:x\in \rr^d \}$ are equivalent to the underlying Lebesgue measure.
Thus, $P$ is regular.  

Moreover, the following minorization condition holds for the STM scheme \eqref{stm}.

\begin{proposition} \label{prop-minor}
Let Assumptions \ref{A1}, \ref{A2}, and \ref{A3} hold, and $\theta\in [1/2,1]$.  
Then for any $\tau \in (0, 1)$ with $L_1 \theta \tau<2$, there exists an $m$-small set $C$ for the Markov chain $(X_n)_{n \in \nn}$ generated by \eqref{stm}, i.e., the minorization condition is satisfied for $P^m$. In addition, the chain is irreducible and strong Feller.
\end{proposition}

\begin{proof}
Due to \cite[Lemma 2.3]{MSH02}, for the minorization condition, it suffices to show the irreducibility and the existence of a jointly continuous density for the transition kernel $P(x, \cdot)$ of $(X_n)_{n\in \nn}$. 

Let $A \in \BB(\rr^d)$ be a non-empty open set.
It follows from Assumption \ref{A3}(ii) that $\hat b$ defined in \eqref{hatb} is continuously differentiable.
Then, by the inverse function theorem, $\hat b$ is an open map that maps open sets to open sets so that the open set $\hat b(A)$ has a positive measure under the nondegenerate Gaussian measure $\mu_{\widetilde{b}(x), \sigma(x)\sigma(x)^\tr \tau}$.
This implies $P(x, A)>0$ and shows the irreducibility of $(X_n)_{n \in \nn}$ in $\rr^d$.
         
To show the existence of a jointly continuous density for the transition kernel $P(x, \cdot)$, we set $A \in \BB(\rr^d)$ and use \eqref{pxa} and change of variables to derive 
\begin{align*}
P(x,A)=& 
\int_{\hat b(A)} \frac{\exp\{ - (y-\widetilde{b}(x))^\tr 
[\sigma(x)\sigma(x)^\tr \tau]^{-1}(y-\widetilde{b}(x))/2\}}
{\sqrt{(2\pi)^d \det{[\sigma(x)\sigma(x)^\tr \tau]}}} \dd y\\
=&\int_{A} p(x,y) \dd y,
\end{align*} 
where for $y \in \rr^d$,
\begin{align} \label{density}
p(x,y):=\frac{\exp\{-(\hat b(y)-\widetilde{b}(x))^\tr 
[\sigma(x)\sigma(x)^\tr \tau]^{-1}(\hat b(y)-\widetilde{b}(x))/2\}}
{\sqrt{(2\pi)^d \det{[\sigma(x)\sigma(x)^\tr \tau] }}} 
\abs{\det J_{\hat b}(y)}.  
\end{align}   
In the change of variables, we used the fact that $\det J_{\hat b} (x) \neq 0$.
Indeed, suppose that $\det J_{\hat b} (x) = 0$, then $0$ must be an eigenvalue of $J_{\hat b}(x)$ with a corresponding eigenvector $v\neq 0$. Notice that by the continuous differentiability of $\hat b$, we have
\[
\hat{b}(x+z) = \hat{b}(x) + J_{\hat b}(x)z + o(|z|), \quad \mbox{when $\rr^d \ni z \to 0$}.
\]
The uniform monotonicity of $\hat b$ shown in the proof of Lemma \ref{lem:well-define} implies that 
\[
\< z, J_{\hat b}(x)z + o(|z|) \> \geq c |z|^2, \quad \mbox{when $z \to 0$},
\]
with $c>0$.
Let $z:= kv$ with $k\to 0$ and we write $o(|z|)$ explicitly as $\psi(z)z$ with $\psi(z)\to 0$ as $z\to 0$. Then, the inequality above becomes
\[
\< kv, J_{\hat b}(x)kv + \psi(kv)kv \> \geq c k^2 |v|^2, \quad \mbox{when $k \to 0$}.
\] 
As $v$ is an eigenvector of $0$, it follows that $\psi(kv)\geq c>0$ as $k\to 0$, which is a contradiction because we assume $\psi(kv)\to 0$ as $k\to 0$.

Therefore, $P(x, \cdot)$ possesses a density $p(x, \cdot)$ given by \eqref{density} with respect to the Lebesgue measure in $\rr^d$.
For each $x_0, y_0\in \rr^d$, 
$p(x_0,y)$ is continuously differentiable in $y$ in the neighborhood 
of $x_0$ (thus locally uniformly continuous) and 
$\lim_{x\to x_0}p(x, y_0)=p(x_0, y_0)$ exists.
Once the uniform continuity on one variable is obtained, we can interchange the order of limit and able to claim  
\begin{align*}
\lim_{\substack{x \to x_0 \\ y \rightarrow y_0}}
p(x, y)\xlongequal{exists} \lim_{x\to x_0}\lim_{y \to y_0} p(x,y) =
\lim_{y \to y_0}\lim_{x\to x_0} p(x,y),
\end{align*}
which shows the joint continuity of the density $p$ defined above.  
\end{proof}

 \begin{remark} 
 It should be pointed out that the (one-step) irreducibility may be lost for the following 
Milstein-type of scheme:
\begin{align*} 
Y_{j+1} & =Y_j+(1-\theta) b(Y_j) \tau
+ \theta b(Y_{j+1}) \tau 
+\sigma(Y_j) \delta_j W   \\
& \quad + \sigma'(Y_j) \sigma(Y_j) 
\Big[\int_{t_j}^{t_{j+1}} (W(r)-W(t_j)) {\rm d}W_r \Big], \quad j \in \nn,  
\end{align*}
under slightly more regularity and commutativity conditions on $\sigma$ as in \cite{WL19}.
Indeed, for $d=m=1$, it is clear that the above Milstein scheme is equivalent to 
\begin{align*} 
Y_{j+1} & =Y_j+(1-\theta) b(Y_j) \tau
+ \theta b(Y_{j+1}) \tau 
+\sigma(Y_j) \delta_j W \\
& \quad  +  \frac12 \sigma'(Y_j) \sigma(Y_j) [(\delta_j W)^2-\tau], 
\quad j \in \nn.
\end{align*}  
If $\sigma'$ is invertible, for any $x \in \rr^d$, it is not difficult to show the existence of a non-empty open set $A \subset \BB(\rr)$ such that $P(x, A)=0$.
  \end{remark}

In combination with the Lyapunov structures developed in Theorem \ref{tm-lya} and Corollary \ref{cor-lya-mid} and the minorization condition derived in Proposition \ref{prop-minor}, we have the following uniform ergodicity of the STM scheme \eqref{stm} with $\theta=1/2$ and geometric ergodicity of the STM scheme \eqref{stm} with $\theta \in (1/2, 1]$.

\begin{theorem} \label{tm-erg}
      Let Assumptions \ref{A1}, \ref{A2}, and \ref{A3} hold. Then for any $\tau \in (0, 1)$ with $L_1 \theta \tau<2$, the STM scheme \eqref{stm} is uniquely ergodic with $\theta=1/2$, and 
it is geometric ergodic for $\theta \in (1/2, 1]$ with respect to the unique invariant measure $\pi_\tau^\theta$, i.e., there exists $r \in (0,1)$ and $b \in (0,\infty)$ such that for all measurable $\phi$ with $|\phi|\leq V$,
\begin{align*}
\abs{\ee \phi(X_n^{x_0})- \int_{\rr^d} \phi(y) \pi_\tau^\theta(\dd y)}
\leq b r^n V(x_0), \quad x_0 \in \rr^d.
\end{align*}  
\end{theorem}

\begin{proof}
  We shall apply the two results we stated at the end of Subsection \ref{sec2.2}. We have shown in Proposition \ref{prop-minor} that for $\theta\in [1/2,1]$, the chain generated by the STM scheme satisfies the minorization condition and is also irreducible and strong Feller. In Corollary \ref{cor-lya-mid}, we have shown the Lyapunov structure \eqref{lya1} for $\theta=1/2$. This shows the unique ergodicity. In Theorem \ref{tm-lya}, we have shown the Lyapunov structure \eqref{lya2} for $\theta\in (1/2,1]$. This shows the geometric ergodicity.
\end{proof}

\section{Unique Ergodicity of Galerkin-based Full Discretizations for Monotone SPDEs}
\label{sec4}

In this section, we shall apply a similar methodology as in the previous section to derive the unique ergodicity of Galerkin-based full discretizations for the following SPDE:
\begin{align}\label{see-fg} 
& {\rm d} X(t, \xi)  
=(\Delta X(t, \xi)+f(X(t, \xi))) {\rm d}t
+g(X(t, \xi)) {\rm d}W(t, \xi), \tag{SPDE}
\end{align}
under (homogenous) Dirichlet boundary condition (DBC) $X(t, \xi)=0$, $(t, \xi) \in \rr_+\times \partial \OOO$, with initial value condition $X(0, \xi)=X_0(\xi)$, $\xi \in \OOO$, 
where the physical domain $\OOO \subset \rr^d$ ($d=1,2,3$) is a bounded domain with a smooth boundary $\partial \OOO$ or a convex domain with a polygonal boundary. 
Here, $f$ is assumed to be monotone-type with polynomial growth, and $g$ satisfies the usual Lipschitz condition in an infinite-dimensional setting; see Assumptions \ref{ap-f} and \ref{ap-g}, respectively.
It is clear that \eqref{see-fg} includes the following stochastic Allen--Cahn equation (SACE), arising from phase transition in materials science by stochastic perturbation, as a special case:
\begin{align} \label{ac}
{\rm d} X(t, \xi)=  \Delta X(t, \xi)  {\rm d}t + \epsilon^{-2} (X(t, \xi)-X(t, \xi)^3) {\rm d}t
+ {\rm d}W(t, \xi), \tag{SACE}
\end{align}
under DBC, where the positive index $\epsilon \ll 1$ is the interface thickness; see, e.g., \cite{BGJK23}, \cite{BJ19}, \cite{BCH19}, \cite{CH19}, \cite{GM09}, \cite{LQ20}, \cite{LQ21}, \cite{OPW23}, and references therein.

\subsection{Preliminaries and Galerkin-based Full Discretizations}
\label{sec4.1}

We first introduce some notations and main assumptions in the infinite-dimensional case.
 
Denote by $\|\cdot\|_H$ and $\<\cdot, \cdot\>_H$ the inner product and norm, respectively, in $H:=L^2(\OOO)$. We usually omit the subscript $H$ if there is no confusion about the notations used for SODEs in Section \ref{sec2}.      
For $\theta=-1$ or $1$, we use $(\dot H^\theta=\dot H^\theta(\OOO), \|\cdot\|_\theta)$ to denote the usual Sobolev interpolation spaces, respectively; the dual between $\dot H^1$ and $\dot H^{-1}$ are denoted by $_1\<\cdot, \cdot\>_{-1}$. 
  
Let ${\bf Q}$ be a self-adjoint and positive definite linear operator on $H$. 
Denote $H_0:={\bf Q}^{1/2} H$ and by
$(\LL_2^0:=HS(H_0; H), \|\cdot\|_{\LL_2^0})$ the space of Hilbert--Schmidt operators from $H_0$ to $H$.     
The driven process $W$ in Eq.~\eqref{see-fg} is an $H$-valued $\bf Q$-Wiener process on $(\Omega, \FFF, \mathbb{F}, \pp)$, which has the Karhunen--Lo\`eve expansion $W(t)=\sum_{k\in \nn} \sqrt{q_k}g_k \beta_k(t)$, $t\geq 0$. 
Here $\{g_k\}_{k\in\nn_+}$ are the eigenvectors of $\bf Q$ and form an orthonormal basis of $H$, with respect to eigenvalues $\{q_k\}_{k=1}^\infty$, and $\{\beta_k\}_{k\in\nn_+}$ are mutually independent 1-D Brownain motions in $(\Omega, \FFF, \mathbb{F}, \pp)$; see, e.g., \cite[Section 2.1]{LR15} for more details.
We only focus on trace-class noise, i.e., $\bf Q$ is a trace-class
operator or equivalently, ${\rm Tr}({\bf Q}):=\sum_{k=1}^\infty q_k<\infty$; we also refer to \cite{Liu25} for the white noise case using a different method.

Our main conditions on the coefficients of \eqref{see-fg} are the following two assumptions.

\begin{assumption} \label{ap-f}
There exist scalars $K_i \in \rr$, $i=1,2,3,4,5$, and $q \ge 1$ such that 
\begin{align} 
& (f(\xi)-f(\eta)) (\xi-\eta) \le K_1 (\xi-\eta)^2,
\quad \xi, \eta \in \rr, \label{f-mon} \\
& f(\xi) \xi \le  K_2 |\xi|^2 + K_3,
\quad \xi \in \rr,  \label{f-coe} \\
& |f'(\xi)|\le K_4 |\xi|^{q-1}+K_5,\quad \xi \in \rr. \label{con-f'}
\end{align}
\end{assumption}

Throughout, we assume that $q \ge 1$ when $d=1,2$ and $q \in [1,3]$ when $d=3$, 
so that the Sobolev embeddings $\dot H^1 \subset L^{2q}(\OOO) \subset H$ hold.
Then we can define the Nemytskii operator $F: \dot H^1 \rightarrow \dot H^{-1}$ associated with $f$ by
\begin{align} \label{df-F}
F(x)(\xi):=f(x(\xi)),\quad x \in \dot H^1,\ \xi \in \OOO.
\end{align}
It follows from the monotone condition \eqref{f-mon} and the coercive condition \eqref{f-coe} that the operator $F$ defined in \eqref{df-F} satisfies 
\begin{align} 
_1\<x-y, F(x)-F(y)\>_{-1}  \le K_1 \|x-y\|^2, & \quad x,y \in \dot H^1, \label{F-mon}  \\
_1\<x, F(x)\>_{-1} \le  K_2 \|x\|^2+K_3, & \quad x \in \dot H^1. \label{F-coe} 
\end{align}
The inequality \eqref{F-coe}, in combination with the Poincar\'e inequality that 
\begin{align} \label{poin}
\|\nabla x\|^2 \ge \lambda_1 \|x\|^2, \quad x \in \dot H^1, 
\end{align}
where $\lambda_1$ denotes the first eigenvalue of $-\Delta$ in $H$, implies that  
\begin{align} 
_1\<x, \Delta x+F(x)\>_{-1} 
& \le -(\lambda_1- K_2) \|x\|^2 +K_3, \quad x \in \dot H^1. \label{F-coe-0}
\end{align}

Denote by $G: H \rightarrow \LL_2^0$ the Nemytskii operator associated with $g$:
\begin{align} \label{df-G}
G(x) g_k(\xi):=g(x(\xi)) g_k(\xi), \quad x \in H,~ k \in \nn,~ \xi \in \OOO.
\end{align}
The following Lipschitz continuity and linear growth conditions are our main conditions on the diffusion operator $G$ defined in \eqref{df-G}.

\begin{assumption} \label{ap-g}
There exist positive constants $K_6$, $K_7$, and $K_8$ such that 
\begin{align}  
\|G(x)-G(y)\|_{\LL_2^0}^2 \le K_6 \|x-y\|^2, & \quad x, y \in H, \label{g-lip} \\
\|G(x) \|^2_{\LL_2^0} \le K_7 \|x\|^2+K_8, & \quad x \in H. \label{g-gro} 
\end{align} 
\end{assumption}

With these preliminaries, \eqref{see-fg} is equivalent to the following infinite-dimensional stochastic evolution equation:
\begin{align} \label{see}
{\rm d}X(t)=(\Delta X(t)+F(X(t))) {\rm d}t+G(X(t)) {\rm d}W, \quad  t \ge 0, \tag{SEE}
\end{align}
where the initial datum $X(0) \in H$ is assumed to vanish on the boundary $\partial \OOO$ of the physical domain throughout the present paper.
 Under the above Assumptions \ref{ap-f} and \ref{ap-g}, the authors in \cite[Theorem 4.2.4]{LR15} showed the existence and uniqueness with moments estimate of the $(\FFF_t)_{t \ge 0}$-adapted solution to Eq.~\eqref{see-fg} or the equivalent Eq.~\eqref{see}.

We also need to assume the nondegeneracy of $G$ in the following sense.

\begin{assumption} \label{ap-fg} 
For any $x \in H$, $G(x) G(x)^*$ is a positive definite operator in $\LL(H)$, i.e., 
\[
\< G(x) G(x)^* h, h\> > 0, \quad \mbox{for all non-zero $h\in H$}.
\]
\end{assumption}

\begin{remark}\label{rk-non}
\begin{enumerate}[label=(\roman*)]
\item
  Together with the assumption that $\bf Q$ is self-adjoint and positive definite we made at the beginning of Section \ref{sec4}, Assumption \ref{ap-fg} is equivalent to the positive definiteness of $[G(x){\bf Q}^{1/2}][G(x){\bf Q}^{1/2}]^*$ for any $x\in H$. 
Indeed, we observe that if the operator is of the form ``$AA^*$'', then $\< A A^* u, u\> >0$ if and only if $A^* u\neq 0$, since
$\< A A^* u, u\> = \<  A^* u, A^* u\> = \|A^* u\|^2$.
  Hence we have
  \begin{itemize}
    \item $G(x)G(x)^*$ is positive definite if and only if $G^* u \neq 0$ for all non-zero $u \in H$;
    \item $[G(x){\bf Q}^{1/2}][G(x){\bf Q}^{1/2}]^*$ is positive definite if and only if $[G(x){\bf Q}^{1/2}]^* u \neq 0$ for all non-zero $u \in H$; and moreover,
    \item $\bf Q$ is positive definite if and only if ${\bf Q}^{1/2}u \neq 0$ for all non-zero $u \in H$ (recall that if $\bf Q$ is self-adjoint and positive definite, ${\bf Q}^{1/2}$ is also self-adjoint and positive definite).
  \end{itemize}

Let $u$ be a non-zero element of $H$. It is now clear that 
${\bf Q}^{1/2} G(x)^* u \neq 0$ if and only if 
$G(x)^* u \neq 0$ 
  by the positive definiteness of $\bf Q$. This shows the equivalence between the two conditions. 
  
  \item
  (Readers may first read the notations we introduced below and then come back to the remark) Under Assumption \ref{ap-fg}, the finite-dimensional projection of the operator is also nondegenerate: for any $x \in H$, 
  \[ 
  [\PP_h G(x){\bf Q}^{1/2}][\PP_h G(x){\bf Q}^{1/2}]^*
  \] is a positive definite operator in $\LL(V_h)$. To see this, it is enough to see that $\PP_h^* u \neq 0$ for non-zero $u\in V_h$. Indeed, if $\PP_h^* u =0$ for $u\in V_h$, then $_{1}\<e_k, \PP_h^* u\>_{-1}= \<\PP_h e_k, u\> = \<e_k, u\> =0$ for all $\{e_k\}_{k=1}^K$ the orthonormal basis of $V_h$. This implies $u=0$. 
  \end{enumerate}
  \end{remark}

To introduce the Galerkin-based fully discrete scheme, let $h\in (0,1)$, $\TT_h$ be a regular family of quasi-uniform partitions of $\OOO$ with maximal length $h$, and $V_h \subset \dot H^1$ be the space of continuous functions on $\bar \OOO$ which are piecewise linear over $\TT_h$ and vanish on the boundary $\partial \OOO$.
Denote by $\Delta_h: V_h \rightarrow V_h$ and $\PP_h: \dot H^{-1} \rightarrow V_h$ be the discrete Laplacian and generalized orthogonal projection operators, respectively, defined by 
\begin{align*}  
\<v^h, \Delta_h x^h\> & =-\<\nabla x^h, \nabla v^h\>,
\quad x^h, v^h\in V_h,  \\
\<v^h, \PP_h z\> & =_1\<v^h, z\>_{-1},
\quad z \in \dot H^{-1},\ v^h\in V_h.  
\end{align*}

We discretize Eq.~\eqref{see} in time with a Drift-Implicit Euler (DIE) scheme, which can be viewed as a particular case of the STM \eqref{stm} with $\theta=1$ in infinite dimension, and in space with a Galerkin approximation.
Then the resulting fully discrete, DIE Galerkin (DIEG) scheme of \eqref{see} is to find a $V^h$-valued discrete process $\{X^h_j:\ j \in \nn\}$ such that
\begin{align}\label{die-g} \tag{DIEG} 
&X^h_{j+1}
=X^h_j+\tau \Delta_h X^h_{j+1}
+\tau \PP_h F(X^h_{j+1})
+\PP_h G(X^h_j) \delta_j W,  
\end{align} 
$j \in \nn$, starting from $X^h_0=\PP_h X_0$, which had been widely studied; see, e.g., \cite{FLZ17} and \cite{LQ21}.

\begin{remark} 
One can also consider a spectral Galerkin version of the DIEG scheme \eqref{die-g} such that all results in this section will be valid, where $\Delta_h$ and $\PP_h$ in \eqref{die-g} are replaced by spectral Galerkin approximate Laplacian operator $\Delta_N: V_N \rightarrow V_N$ and the generalized orthogonal projection operators $\PP_N: \dot H^{-1} \rightarrow V_N$, respectively: 
\begin{align*}  
\<\Delta_N u^N, v^N\> & =-\<\nabla u^N, \nabla v^N\>,
\quad u^N, v^N \in V_N,  \\
\<\PP_N u, v^N\> & =_1\<v^N, u\>_{-1},
\quad u\in V^*,\ v^N \in V_N,
\end{align*}
where $V_N$ denotes the space spanned by the first $N$-eigenvectors of the Dirichlet Laplacian operator which vanish on $\partial \OOO$ for $N \in \nn_+$. 
\end{remark}

To derive the unique solvability of the DIEG scheme \eqref{die-g}, we define $\hat F: V_h \to V_h$ by
  \begin{align} \label{hatF}
  \hat F(x) := ({\rm Id}- \tau \Delta_h) x- \tau \PP_h F(x), \quad x \in V_h. 
  \end{align}
Here and what follows, ${\rm Id}$ denotes the identity operator in various Hilbert spaces if there is no confusion.
  Then \eqref{die-g} becomes
  \begin{equation*}
      \hat F(X_{j+1}^h) = X_j^h 
      + \PP_h G(X_j^h) \delta_j W, \quad j \in \nn.
  \end{equation*}

\begin{lemma} \label{lm-hatF}
  Let Assumptions \ref{ap-f} hold with $(K_1-\lambda_1) \tau<1$. 
  Then $\hat F: V_h \to V_h$ defined in \eqref{hatF} is bijective so that the DIE scheme \eqref{die-g} can be uniquely solved and 
  $(X_{j}^h)$ is a $V_h$-valued homogenous Markov chain.
  Morevoer, $\hat F$ is an open map, i.e. for each open set $A\in \BB(V_h)$, $\hat F(A)$ is also an open set in $\BB(V_h)$.
\end{lemma}

\begin{proof}   
We have for all $x,y \in V_h$, 
  \begin{align*}
      \<x-y, \hat F(x)-\hat F(y)\> 
&= \norm{x-y}^2 + \tau  \norm{\nabla(x-y)}^2 -\tau  {}_1\<x-y,F(x)-F(y)\>_{-1} \\
      &\geq (1-K_1 \tau)\norm{x-y}^2 + \tau  \norm{\nabla(x-y)}^2 \\
      & \geq (1-K_1 \tau + \lambda_1 \tau) \norm{x-y}^2
  \end{align*}
  where we have used the Poincar\'e inequality \eqref{poin} in the last inequality. This shows the bijectivity of $\hat F$ (since $V_h$ is finite-dimensional, monotone uniform fixed point theorem holds). The homogenous Markov property is clear as the SODE case; see the proof of Lemma \ref{lem:well-define}. 

It remains to show that $\hat F$ is an open map.
From the uniform monotonicity and Cauchy--Schwarz inequality,  
we obtain 
  \begin{align}\label{eq-mon}
    \| \hat F(x) - \hat F(y)\| \geq (1-K_1 \tau + \lambda_1 \tau) \|x-y\|,   \quad \forall~x,y \in V_h.
  \end{align}
  As we explained below, $\hat F$ is an open map. From \eqref{eq-mon}, we see that
  \begin{align}\label{ball inclusion}
    B( \hat F(x), r) \subset \hat F(B(x, r/C_0)), \quad ~\forall ~r>0,
  \end{align}
  where $C_0:= 1-K_1 \tau + \lambda_1 \tau>0$. 
  Fix an open set $A\in \BB(H)$. Our target, $\hat F(A)$ being open, means that for each point $x\in V_h$, there exists an open ball $B(\hat F(x), r_0)$ with $r_0>0$ such that $B(\hat F(x), r_0)\subset \hat F(A)$. Due to the inclusion \eqref{ball inclusion}, such ball exists if $B(x, r_0/C_0)\subset A$ exists. But this is guaranteed because $A$ is open. 
\end{proof}

\begin{remark} 
  We can use the same idea as above to show the $\hat b$ defined in \eqref{hatb} in previous sections is also an open map by its uniform monotonicity, thus relieving the assumption of continuous differentiability in Assumption \ref{A3} for the proof of regularity (see the discussion above Proposition \ref{prop-minor}) of transition kernels of \eqref{stm}. However, it cannot be relieved in Proposition \ref{prop-minor} when the joint continuity of the transition probability density is needed.
\end{remark}

\subsection{Lyapunov structure of DIEG}

We have the following Lyapunov structure for the DIEG scheme \eqref{die-g}; see \cite[Lemma 3.1]{Liu24} for a similar uniform moments estimate for the temporal semi-discretization DIE scheme. 
Whether there exists a similar Lyapunov structure as that of \eqref{stm} in the SODE case for general $\theta \in [1/2, 1)$ is unknown.

 \begin{theorem} \label{tm-lya-dieg}
  Let Assumptions \ref{ap-f} and \ref{ap-g} hold with $K_2+K_7/2<\lambda_1$.
Then for any $\tau \in (0, 1)$ with $(K_1-\lambda_1) \tau<1$ and $\epsilon \in (0, \lambda_1-K_2)$, there exists a Lyapunov function $V: V_h \to [1,\infty)$ defined by 
\begin{equation}\label{lya-spde}
          V(x)=\|x\|^2 + \frac{2 \epsilon \tau}{[1 + 2 (\lambda_1-K_2-\epsilon) \tau] \lambda_1}
          \|\nabla x\|^2 +1, \quad x \in V_h, 
      \end{equation}
       such that \eqref{lya2} holds for \eqref{die-g}, with $\rho=\frac{1+K_7 \tau}{1 + 2 (\lambda_1-K_2-\epsilon) \tau} \in (0, 1)$ and $\kappa=\frac{(2K_3+K_8) \tau}{1 + 2 (\lambda_1-K_2-\epsilon) \tau} + (1-\rho)> 0$, i.e., for $n\in \nn$, 
      \begin{equation}\label{lya-die-g}
        \begin{split}
          \ee [V(X_{n+1}^h)\mid \FFF_{n}] \leq &
          \frac{1+K_7 \tau}{1 + 2 (\lambda_1-K_2-\epsilon) \tau} V(X_n^h) \\
          &\quad +  \frac{[2 (\lambda_1-K_2+K_3-\epsilon) -K_7+K_8] \tau}{1 + 2 (\lambda_1-K_2-\epsilon) \tau}.
        \end{split}
      \end{equation} 
  \end{theorem}

\begin{proof}  
For simplicity, set $F_k=F(X_k^h)$ and $G_k=G(X_k^h)$ for $k \in \nn$.
Testing \eqref{die-g} using $\<\cdot, \cdot\>$-inner product with $X_{k+1}^h$, using the elementary equality 
\begin{align*}
2 \<x-y, x\> =\|x \|^2- \|y\|^2+\|x-y\|^2,
\quad x, y \in H,
\end{align*}  
and integration by parts formula, we have  
\begin{align} \label{sta-uk0}
& \|X_{k+1}^h\|^2 - \|X_k^h\|^2 +  \|X_{k+1}^h-X_k^h\|^2
+ \frac{2 \varepsilon}{\lambda_1}  \|\nabla X_{k+1}^h\|^2 \tau \nonumber \\
& = - 2(1-\frac{\varepsilon}{\lambda_1}) \|\nabla X_{k+1}^h\|^2 \tau +2 ~ _1\<X_{k+1}^h, F_{k+1}\>_{-1} \tau \nonumber \\
& \quad + 2 \<X_{k+1}^h-X_k^h, G_k \delta_k W\>
+ 2 \<X_k^h, G_k \delta_k W\>,
\end{align}   
for any positive $\varepsilon$ with $\varepsilon<\lambda_1$; here and in the rest of the paper, $\varepsilon$ denotes an arbitrarily small positive constant which would differ in each appearance.
Using \eqref{F-coe}, \eqref{poin}, and the elementary inequality $2\<a,b\> \leq \|a\|^2+\|b\|^2$ that holds in general inner product spaces, we have
\begin{align*} 
&  \|X_{k+1}^h\|^2 - \|X_k^h\|^2  
+ \frac{2 \varepsilon}{\lambda_1} \|\nabla X_{k+1}^h\|^2 \tau  \\
& \le - 2 (\lambda_1-K_2- \varepsilon) \|X_{k+1}^h\|^2 \tau 
+ \|G_k \delta_k W\|^2 + 2\<X_k^h, G_k \delta_k W\>+ 2 K_3 \tau. 
\end{align*}  
It follows that
\begin{align*}  
& (1+2 (\lambda_1-K_2-\varepsilon) \tau) \|X_{k+1}^h\|^2 
+ \frac{2 \varepsilon}{\lambda_1} \|\nabla X_{k+1}^h\|^2 \tau \\
& \le \|X_k^h\|^2+\|G_k \delta_k W\|^2 
+ 2\<X_k^h, G_k \delta_k W\>+2 K_3 \tau.
\end{align*}  

Taking the conditional expectation $\ee [\cdot\mid \FFF_n]$ on both sides, noting the fact that both $X_k^h$ and $G_k$ are independent of  $\delta_k W$, using It\^o isometry and \eqref{g-gro}, we get
\begin{align*}  
& (1+2 (\lambda_1-K_2-\varepsilon) \tau) \ee [\|X_{k+1}^h\|^2 \mid \FFF_n] 
+ \frac{2 \varepsilon}{\lambda_1}\tau \ee [\|\nabla X_{k+1}^h\|^2 \mid \FFF_n]  \\
& \le (1+K_7 \tau)  \|X_k^h\|^2 + (2 K_3+K_8) \tau,
\end{align*} 
from which we obtain \eqref{lya-die-g}.  
\end{proof}

  \begin{theorem} \label{tm-dieg}
Let Assumptions \ref{ap-f}, \ref{ap-g}, and \ref{ap-fg} hold with $K_2<\lambda_1$.
Then the DIEG scheme \eqref{die-g} is uniquely ergodic for any $h \in (0, 1)$ and $\tau \in (0, 1)$ with $(K_1- \lambda_1) \tau<1$.
  \end{theorem}

  \begin{proof}  
  With the help of the Lyapunov structure derived in Theorem \ref{tm-lya-dieg}, it suffices to show the regularity of the transition kernel $P^h$ associated with the DIEG scheme \eqref{die-g} defined as 
  \begin{align*}
  P^h(x,A)=\pp(X_{n+1}^h \in A\mid X_n^h=x), \quad x \in V_h, ~ A \in \BB(V_h).
  \end{align*}

By the Markov property of \eqref{die-g} (similar to what we did in \eqref{pxa}) and calculations of mean and covariance, we have 
\begin{align} \label{pxa+}
    P^h(x, A) 
    = \mu_{x, [\PP_h G(x) {{\bf Q}^{1/2}}] [\PP_h G(x) {{\bf Q}^{1/2}}]^* \tau}(\hat F(A)),  \quad x \in V_h, ~ A \in \BB(V_h), 
\end{align} 
as $x + \PP_h G(x) \delta_n W \sim \mathcal N(x, [\PP_h G(x) {{\bf Q}^{1/2}}] [\PP_h G(x) {{\bf Q}^{1/2}}]^* \tau)$, where $\hat F$ is defined in $\eqref{hatF}$.  
As noted in Remark \ref{rk-non}(ii), $[\PP_h G(x) {{\bf Q}^{1/2}}] [\PP_h G(x) {{\bf Q}^{1/2}}]^* \tau$ is nondegenerate, so that the family of Gaussian measures $\{\mu_{x, [\PP_h G(x) {{\bf Q}^{1/2}}] [\PP_h G(x) {{\bf Q}^{1/2}}]^* \tau}: x\in V_h\}$ are all equivalent by applying Feldman--Hajek theorem in finite dimensional space $V_h$.
This shows that $P^h$ is regular.
  \end{proof}

 \begin{remark}
One can show that the estimate \eqref{lya-die-g} holds for the temporal DIE scheme with the same function $V$ defined in \eqref{lya-spde}, which is indeed a Lyapunov function in $H$ by the compact embedding $\dot H^1 \subset H$.
However, the regularity or the strong Feller property in the infinite-dimensional case is unknown, so one cannot conclude the unique ergodicity of the temporal semi-discrete DIE scheme.  
 \end{remark}

Applying the above result in Theorem \ref{tm-dieg}, we have the following unique ergodicity of the DIEG scheme \eqref{die-g} applied to the \eqref{ac}.

  \begin{theorem} \label{tm-ac}
  For any $\epsilon>0$, $h \in (0, 1)$ and $\tau \in (0, 1)$ with $(\epsilon^{-2}-\lambda_1) \tau<1$, the numerical schemes \eqref{die-g} and \eqref{die-g} for \eqref{ac} are both uniquely ergodic. 
  \end{theorem}

  \begin{proof}  
  We just need to check the conditions in Assumptions \ref{ap-f}, \ref{ap-g}, and \ref{ap-fg} hold with $K_2<\lambda_1$ in the setting of \eqref{ac} with $q=3$, $f(x)=\epsilon^{-2}(x-x^3)$, $x \in \rr$, and $G(x)={\rm Id}$ for all $x \in H$.
  
  As $G(x)={\rm Id}$ for all $x \in H$, $G(x) G(x)^*= {\rm Id}$ is positive definite so that Assumption \ref{ap-fg} holds.
  The validity of \eqref{f-mon} and \eqref{con-f'} in Assumption \ref{ap-f} and \eqref{g-lip} and \eqref{g-gro} in Assumption \ref{ap-g} in the setting of Eq.~\eqref{ac} were shown in \cite[Section 4]{Liu24}: $K_1=\epsilon^{-2}$, $K_4=2\epsilon^{-2}$ and $K_5=\epsilon^{-2}$, $K_6=K_7=0$, $K_8={\rm Trace}({\bf Q})$, the trace of ${\bf Q}$.
  It remains to show \eqref{f-coe} with $K_2<\lambda_1$ and some $K_3>0$.
  Indeed, by Young inequality,
\begin{align*}
\epsilon^{-2}(x-x^3)x
=-\epsilon^{-2} x^4+\epsilon^{-2}x^2
\le -C x^2+C_\epsilon,
\end{align*}     
for any positive $C$ and certain positive constant $C_\epsilon$.
So one can take $K_2$ as any negative scalar and thus \eqref{f-coe} with $K_2<\lambda_1$ and some $K_3>0$.
  \end{proof}

\section{Numerical Experiments}
\label{sec5}

  In this section, we perform two numerical experiments to verify our theoretical results Theorem \ref{tm-erg} in Section \ref{sec3} and Theorem \ref{tm-ac} in Section \ref{sec4}, respectively.

  \subsection{Experiments on STM}
The first numerical test is given to the Eq.~\eqref{sde} with $b(x)=x-x^3$ and $\sigma(x)=\sqrt{x^2+1}$.
Assumptions \ref{A1}, \ref{A2}, and \ref{A3} have been verified in Example \ref{ex1} with $L_1=3$, $L_2=3$, and $L_3=1$ and thus \eqref{stm} is uniquely ergodic for any $\theta\in [1/2,1]$ and $\tau \in (0, 1)$ (so that $L_1 \theta \tau<2$), according to Theorem \ref{tm-erg}. 
The proposed scheme, which is implicit, is numerically solved by utilizing \verb|scipy.optimize.fsolve|, a wrapper around MINPACK's hybrd and hybrj algorithms.
In addition, we take $\tau=0.1$ and choose $\theta=1/2, 3/4, 1$ and initial data $X_0=-5, 5, 15$ to implement the numerical experiments.
 
It is clear from Figure \ref{fig-sde} that the shapes of the empirical density functions plotted by kernel density estimation at $n=5000$ corresponding to $t=500$ are much more similar with the same $\theta$ and different initial data, which indicates the unique ergodicity of \eqref{stm} and thus verifies the theoretical result in Theorem \ref{tm-erg}. 
Indeed, Figure \ref{fig-sde} indicates the strong mixing property (which yields the unique ergodicity) of \eqref{stm}; this verifies the convergence result in Theorem \ref{tm-erg} with $\theta \in (1/2, 1]$.  
Moreover, the limiting measures for different $\theta$ are pretty close, which indicates the uniqueness of the limiting invariant measure for \eqref{stm} even with different $\theta$.
Indeed, these invariant measures are all acceptable approximations of the limiting invariant measure, which is the unique invariant measure of Eq.~\eqref{sde}.
The quantitative convergence analysis between the exact and numerical ergodic measures has been studied in a separate paper by the first author in collaboration with Dr. Xiaoming Wu.

  \begin{figure}[h]
    \centering 
    \includegraphics[width=1\textwidth]{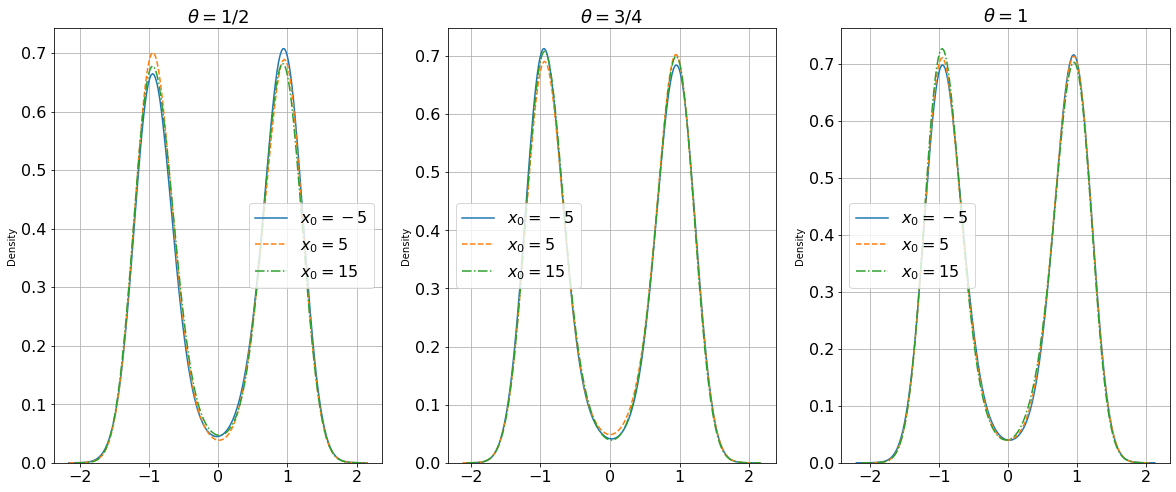}
    \caption{Empirical density of \eqref{stm} for \eqref{sde}}\label{fig-sde}
  \end{figure}
   
\subsection{Experiments on DIEG}
The second numerical test is given to \eqref{ac} in $\OOO=(0, 1)$ with $\epsilon=0.5$.
By Theorem \ref{tm-ac}, \eqref{die-g}, as well as the spectral Galerkin discretization of \eqref{die-g}, is uniquely ergodic for any $\tau \in (0, (4-\pi^2/(\pi^2+1))^{-1})$ (with $(\epsilon^{-2}-\frac{\lambda_1}{\lambda_1+1}) \tau<1$). 
We take $\tau=0.1$ and $N=10$ (the dimension of the spectral Galerkin approximate space), choose $\theta=1/2, 3/4, 1$ and initial data $X_0(\xi)=0, \sin \pi \xi, \sum_{k=1}^{10} \sin k \pi \xi$, $\xi \in (0, 1)$, and approximate the expectation by taking averaged value over $1,000$ paths to implement the numerical experiments.
In addition, we simulate the time averages $\frac{1}{2,000}\sum_{n=1}^{2,000} \ee [\phi(X_n)]$ (up to $n=2,000$ corresponding to $t=200$) by 
  \begin{align}
    \frac{1}{2,000,000}\sum_{n=1}^{2,000} \sum_{k=1}^{1,000} \phi(X_n^k),
  \end{align}
  where $X_n^k$ denotes $n$-th iteration of $k$-th sample path and the test function $\phi$ are chosen to be $\phi(\cdot)=e^{-\|\cdot\|^2}, \sin \|\cdot\|^2, \|\cdot\|^2$, respectively.

  \begin{figure}[h]
    \centering
    \includegraphics[width=\textwidth]{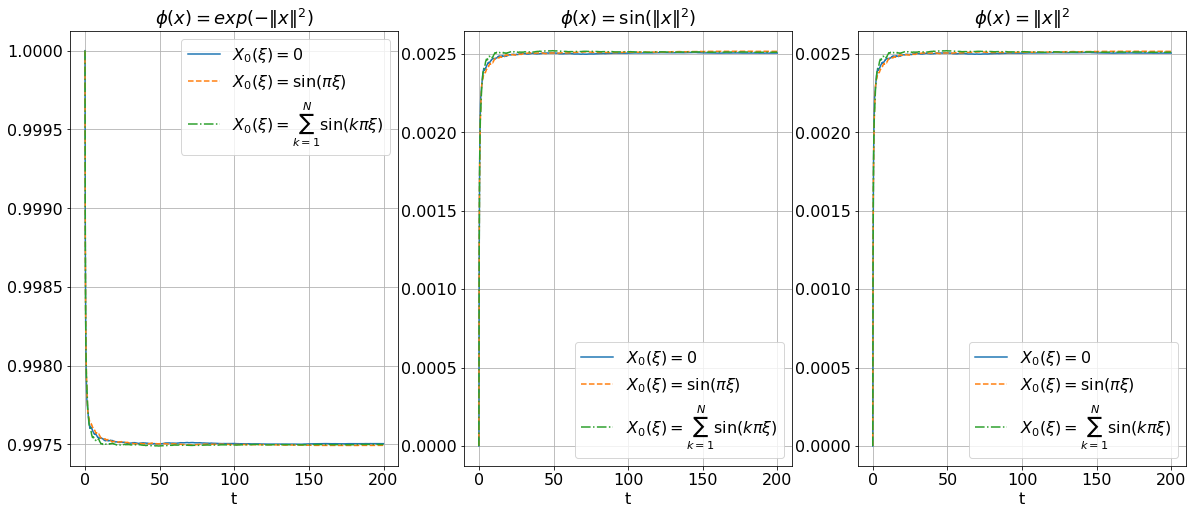}
    \caption{Time averages of \eqref{die-g} for \eqref{ac}}\label{fig-ac}
  \end{figure}

From Figure \ref{fig-ac}, the time averages of \eqref{die-g} with different initial data converge to the same ergodic limit, which verifies the theoretical result in Theorem \ref{tm-ac}.

\section{Discussions}
\label{sec6}

The theoretical loss of geometric ergodicity of \eqref{stm} with $\theta=1/2$ (in Theorem \ref{tm-erg}) is mainly because we can not show the stronger Lyapunov condition \eqref{lya+} at this stage.
However, Figure \ref{fig-sde} indicates that \eqref{stm} with $\theta=1/2$ is also strongly mixing.
This motivates our conjecture that when $\theta=1/2$, or even for $\theta\in (0,1/2]$, \eqref{lya2} also holds with some $\rho\in (0,1)$ and $\kappa\geq 0$ so that it is also geometrically  ergodic. 
The conjecture will be investigated in future research.

\section*{Declarations}

{\bf Conflict of interest}. The authors have no competing interests to declare relevant to this article's content.
The data are available from the corresponding author on reasonable request.

\section*{Acknowledgements}

We thank the anonymous referees for their helpful comments and suggestions.
The first author is supported by the National Natural Science Foundation of China, No. 12101296, Basic and Applied Basic Research Foundation of Guangdong Province, No. 2024A1515012348, and Shenzhen Basic Research Special Project (Natural Science Foundation) Basic Research (General Project), Nos. JCYJ20220530112814033 and JCYJ20240813094919026.

  \bibliographystyle{amsplain}
  \bibliography{bib.bib}

\end{document}